\documentclass[10pt]{amsart}
\usepackage{graphicx}
\usepackage{amssymb}
\usepackage{mathrsfs}
\usepackage{epstopdf}

\theoremstyle{definition} 
\theoremstyle{plain}      \newtheorem{lemma}{Lemma}
\theoremstyle{plain}      
\theoremstyle{plain}      \newtheorem{corollary}{Corollary}
\theoremstyle{plain}      
\theoremstyle{plain}      \newtheorem{proposition}{Proposition}
\theoremstyle{plain}      \newtheorem{theorem}{Theorem}
\theoremstyle{remark}     
\theoremstyle{remark}     
\theoremstyle{plain}

\title[Exact scaling]{Exact scaling in the expansion-modification system}
\author{R. Salgado-Garc\'ia}
\address{Facultad de Ciencias, Universidad Aut\'onoma del Estado de Morelos, 
Avenida Universidad 1001, Colonia Chamilpa, Cuernavaca C.P. 62209, Morelos, Mexico.}

\author{E. Ugalde}
\address{Instituto de F\'isica, Universidad Aut\'onoma de 
San Luis Potos\'i, Avenida Manuel Nava 6, Zona Universitaria, 
78290 San Luis Potos\'\i , M\'exico.}

\date{\today}                                         

\begin{document}

\begin{abstract}
This work is devoted to the study of the scaling, and the consequent power-law behavior, 
of the correlation function in a mutation-replication model known as 
the expansion-modification system. The latter is a biology inspired random substitution 
model for the genome evolution, which is defined on a binary alphabet and depends on 
a parameter interpreted as a \emph{mutation probability}. We prove that the time-evolution 
of this system is such that any initial measure converges towards a unique stationary one 
exhibiting decay of correlations not slower than a power-law.  We  then prove, for a 
significant range of mutation probabilities, that the decay of correlations indeed follows 
a power-law with scaling exponent smoothly depending on the mutation probability. Finally 
we put forward an argument which allows us to give a closed expression for the 
corresponding scaling exponent for all the values of the mutation probability. Such 
a scaling exponent turns out to be a piecewise smooth function of the parameter.
\end{abstract}

\maketitle

\bigskip

\section{Introduction.}

\subsection{} In recent years, several models have been introduced  (such as $n$-step 
Markov chains or  hidden Markov chains, among 
others~\cite{li1997study,ma2008infinite,saakian2008,sobottka2011}) 
to describe the evolution of nucleotide sequences as well as the patterns and 
correlations occurring in the genome.  In this paper we are concerned with one of those 
models, proposed by W. Li~\cite{li1989spatial}, which consists of a sequence (or chain) 
of symbols that evolve according to a given discrete-time stochastic dynamics. Such a 
dynamics captures the essential processes which are assumed to be responsible of the 
genome evolution: the random expansion and  modification of symbols (hence the name of 
expansion-modification system). Originally introduced as a simple model exhibiting some 
spatial scaling properties, a behavior ubiquitous in natural phenomena~\cite{li1989spatial}, 
it was subsequently used to understand the scaling  properties and the long-range 
correlations found in real DNA 
sequences~\cite{buldyrev1995long,li1991expansion,li1992long,li1997study,peng1992}. 
Recently, the expansion-modification system has also been used to investigate the 
universality of the rank-ordering distributions~\cite{alvarez2010order, 
mekler2009universality}. 

\medskip\subsection{}
From the mathematical point of view, the expansion-modification system belongs to the 
class of random substitution dynamical systems, which attracted some attention in recent 
years for their possible applications in genome evolution studies (see~\cite{2012Koslicki} 
and references therein). A related class of stochastic processes, inspired by randomly 
generated grammars, were formalized and studied by Toom and coworkers (see~\cite{2011Toom}). 
Previously,  Godr\`eche and Luck used random substitutions to study the robustness of 
quasiperiodic structures~\cite{1989GodrecheLuck}, in particular structures associated 
to random perturbations of Fibonacci sequences and Penrose tilings. They observed that 
the Fourier spectrum of the structures thus obtained are of mixed type: they contain 
both singular and continuous parts. Fourier spectra of mixed type appear in structures 
corresponding to random perturbation of quasicrystals. For instance, Zaks~\cite{zaks2001} 
observed mixed spectrum in structures generated by randomized Thue-Morse sequences. 
In our case, the structure generated by the expansion-modification system cannot be seen
as a random perturbation of a quasicrystal, as the corresponding Fourier spectrum turns 
out to be continuous. What the expansion-modification system shares with those randomly 
perturbed quasicrystals is the scaling property and the consequent power law behavior of 
the correlations and the Fourier spectrum. For those randomly perturbed quasicrystals, 
the scaling of the perturbed structure can be very easily deduced from the scaling
already present in the underlying quasicrystal, by means of obvious recurrence relations 
derived from the inflation rules. As we will show, the scaling in the 
expansion-modification system derives from recurrence relations implied by the
underlying dynamics, and contrary to the quasicrystal case, these recurrence relations 
grow in complexity in such a way that its treatment demands the implementation of 
nontrivial techniques. The rigorous study of this scaling behavior and the consequent 
power laws is the main contribution of this work.

\medskip \subsection{} 
The expansion-modification system can be described as follows. Consider the random 
substitution
\begin{eqnarray*}
 0 &\mapsto& \left\{\begin{array}{ll} 1 & \text{ with probability } p,\\  
00 & \text{ with probability } 1-p,                                   
\end{array}\right.\\
 1 &\mapsto& \left\{\begin{array}{ll} 0 & \text{ with probability } p,\\ 
11 & \text{ with probability } 1-p,
\end{array}\right.
\end{eqnarray*}
in the binary set $\{0,1\}$, and extend it coordinate-wise to the set $\{0,1\}^+$ 
of finite binary strings. Starting at time zero with a seed in ${\bf x}^0\in\{0,1\}^+$, 
and iterating the above substitution, we obtain a sequence 
\[{\mathbf x}^0\mapsto {\mathbf x}^1\mapsto \cdots \mapsto {\mathbf x}^{n}\mapsto\cdots \]
of finite strings of non-decreasing length. Since the applied substitution is a random map, 
the sequence we obtain by successive iterations is a random sequence which is nevertheless 
supposed to converge, in a certain statistical sense, to a random string 
${\mathbf x}^\infty$. It is easy to see that the probability of having a finite string 
after infinitely many iterations is zero, therefore it is more convenient to study
the evolution of infinite strings under the infinite extension of the above substitution. 
This is precisely the point of view we will follow throughout this work.

\medskip\subsection{}
The paper is organized as follows. In Section~\ref{sec:themodel} we set up the mathematical
framework where the expansion-modification system is defined.
In Section~\ref{sec:results} we state our main results, which we proof in 
Section~\ref{sec:proofs}.
In Section~\ref{sec:heuristic} we compute a closed expression for the scaling exponent 
which presumably holds in the whole range of mutation probabilities. We finish the paper 
with some final remarks and comments.

\bigskip

\section{The Expansion-Modification Dynamics.}
\label{sec:themodel}
\medskip

\subsection{}
In order to review the expansion-modification system, let us start fixing the relevant 
notation and terminology. Let $X = \{0,1\}^{\mathbb{N}_0}$, which we endow with the 
$\sigma$-algebra generated by the cylinder sets. Elements of $X$ are called 
{\em configurations}, and will be denoted by boldface characters 
like $\mathbf{x}=x_0 x_1\cdots$, with $x_i \in \{0,1\}$. Finite sequences of symbols, 
also called \emph{words}, will be also denoted by boldfaced letters while their size 
will be denoted by $|\cdot|$, \emph{i.e.}, for $\mathbf{v}\in  \{0,1\}^k$ we have 
$|\mathbf{v}| = k$. 
A word $\mathbf{v}\in  \{0,1\}^k$ occurs as \emph{prefix} of $\mathbf{x} \in X$,
which we denote by $\mathbf{v}\sqsubseteq\mathbf{x}$, if 
$\mathbf{v}=x_0x_1\cdots x_{k-1}$. We will also use this notation when $\mathbf{x}\in X$ 
is replaced by a finite word. 
Given a configuration $\mathbf{x}=x_0 x_1\cdots\in X$, and integers $0\leq i < j$,  
we denote with $\mathbf{x}_i^j$ the word $x_ix_{i+1}\cdots x_j$.
Product of words will be understood as concatenation: given two words 
$\mathbf{v}\in  \{0,1\}^k$ and $\mathbf{w}\in \{0,1\}^l$, we let $\mathbf{v}\mathbf{w}$ 
denote the word $\mathbf{u}$ of size $k+l$ satisfying $\mathbf{u}_0^{k-1} = \mathbf{v}$ 
and $\mathbf{u}_{k}^{k+l-1} = \mathbf{w}$. 
Consider $S = \{\mbox{e},\mbox{m}\}^{\mathbb{N}_0}$, where the symbols $\mbox{e}$ 
and $\mbox{m}$ stand for expansion and modification respectively. The space $S$, which we 
will refer to as the \emph{space of substitutions}, is endowed with the $\sigma$-algebra
generated by the cylinder sets as well.
We will use the same convention to denote the elements of $S$, words and concatenation of 
words, as for the symbolic space $X$. 

\medskip

\subsection{} Let us now define the \emph{local substitutions} 
$\mathrm{e}, \mathrm{m} : \{ 0,1\} \to \{ 0,1\}^+:=\cup_{n=1}^\infty\{0,1\}^n$, which are
given by
\begin{eqnarray*}
\mathrm{e}(x)  &=& xx,  \\
\mathrm{m} (x) &=& 1-x.
\nonumber
\end{eqnarray*}
A configuration $\mathbf{s}\in S$ of local substitutions defines the 
\emph{global substitution} $\mathbf{s} : X \to X$ given by
\[
\mathbf{s}(\mathbf{x}) = \prod_{i\in \mathbb{N}_0} s_i(x_{i}).
\]
Here $\prod$ stands for concatenation of words. Notice that $\mathbf{s}$ replaces the 
$i$-th symbol of $\mathbf{x}$ according to the $i$-th local substitution, i.e, if  
$s_i = \mbox{e}$ then $x_i$ is expanded, otherwise $x_i$ is modified. 

\medskip

\subsection{}
The \textit{expansion-modification dynamics} is a random dynamical system 
whose orbits depend on an initial condition and a choice of global substitutions to 
be applied to that initial condition. To be more precise, an initial condition 
$\mathbf{x}\in X$ and a sequence
$\mathbf{s}^0,\mathbf{s}^1,\mathbf{s}^2,\ldots$ of configurations in $S$, define 
the orbit 
$\mathbf{x}^0,\mathbf{x}^1,\mathbf{x}^2,\ldots$ in $X$ where $\mathbf{x}^0:=\mathbf{x}$ 
and $\mathbf{x}^{t+1}=\mathbf{s}^t(\mathbf{x}^{t})$ for each $t>0$. 
At each time step $t$, the global substitution $\mathbf{s}^t$ is randomly chosen 
according to the Bernoulli measure $\mu_p$ such that $\nu_p[\mbox{m}]=p$ and 
$\nu_p[\mbox{e}]=1-p$. The parameter $p\in(0,1)$ is the \emph{mutation probability}.

\medskip\noindent 
In terms of distributions, the expansion-modification system can be defined as follows. 
If $\mu^t$ is the measure according to which the time-$t$ configurations are distributed, 
then the distribution $\mu^{t+1}$ of time-$(t+1)$ configurations is completely determined 
by $\nu_p$ and $\mu^t$ according to the following expression:
\begin{equation}\label{eq:marginals}
\mu^{t+1}\{(\mathbf{x}^{t+1})_0^{\ell}=\mathbf{a}\}=
\sum_{\mathbf{c}\in \{\mathrm{e},\mathrm{m}\}^{\ell+1}}
\sum_{\mathbf{b}\in \{0,1\}^{\ell+1}\atop \mathbf{a} \sqsubseteq
\prod_{i=0}^\ell  s_i(b_i)} \mu^{t}\{(\mathbf{x}^t)_0^{\ell}=\mathbf{b}\}\,
\nu_p\{(\mathbf{s}^t)_0^{\ell}=\mathbf{c}\},
\end{equation}
for each $\ell\in \mathbb{N}_0$ and $\mathbf{a}\in\{0,1\}^{\ell +1}$. 
As mentioned before, $\mathbf{a}\sqsubseteq \mathbf{b}$ means that the word 
$\mathbf{a}$ occurs as a suffix of the word $\mathbf{b}$. Hence, the evolution of 
the $(\ell+1)$-marginal is nothing but a Markov chain. Indeed, considering
the $(\ell+1)$-marginal of a measure $\mu$ as a probability vector of dimension 
$2^{\ell+1}$, the $(\ell+1)$-marginal $\mu_\ell^t$, of the time-$t$ distribution 
is given by matrix product $\mu^t_\ell=\mu^0_\ell\,M_\ell^t$, where
$M_\ell:\{0,1\}^{\ell+1}\times\{0,1\}^{\ell+1}\to[0,1]$ is the 
$2^{\ell+1}\times 2^{\ell+1}$-stochastic matrix given by
\[
M_\ell(\mathbf{a},\mathbf{b})=\sum_{\mathbf{c}\in \{\mathrm{e},\mathrm{m}\}^{\ell+1}  
\atop \mathbf{a}\sqsubseteq\prod_{i=0}^\ell c_i(b_i)  } \nu_p[\mathbf{c}].
\]

\bigskip

\section{Results}
\label{sec:results}
\subsection{} 
Our first result states the existence and uniqueness of the stationary distribution, 
which turns out be a global attractor for the expansion-modification dynamics.

\medskip\noindent
\begin{theorem}[Existence and Uniqueness]\label{theo:asymptotic}
For each $p\in (0,1)$ there exists a unique measure $\mu_p$ on $X$ which is invariant 
under the expansion-modification dynamics. Furthermore, starting from any measure 
$\mu^0$ determining the distribution of the initial conditions, the measure $\mu^t$, 
corresponding to the distribution at time $t$, converges in the *-weak sense to $\mu_p$.
\end{theorem}

\medskip \noindent It is not hard to see that uniqueness does not hold for $p\in\{0,1\}$. 
For $p=0$, when only expansion is possible, each convex combinations of the Dirac measures 
at all-zeros and all-ones, is an admissible invariant distribution.
On the other hand, in the case $p=1$, the dynamic of each initial distributions enters 
a period-two cycle, excepting for the measures that are flip-invariant 
($0\leftrightarrow 1$). In both cases, the asymptotic regime depends on the initial 
distribution.

\medskip \noindent In~\cite{2011Toom} Toom and coworkers prove the existence of invariant
measures for substitution operators similar to, but not including, the 
expansion-modification dynamics. In their case, due to random deletion of words, the 
dynamics cannot be reduced to the action of stochastic matrices over finite length marginals.

\medskip 
\subsection{} The main result of this paper establishes the power-law 
decay of correlations exhibited by the unique stationary measure $\mu_p$. Before stating 
this theorem, let us remind the main definitions.

\medskip \noindent
The \emph{two-sites correlation function}, $C_p:\mathbb{N}_0\to\mathbb{R}$, is given by
\[
C_p(n):=\int_{X} {x}_0 {x}_n\,d\mu_p(\mathbf{x})-
   \left(\int_X {x}_0\,d\mu_p(\mathbf{x})\right) 
   \left(\int_X  {x}_n\,d\mu_p(\mathbf{x})\right),
\]
where ${x}_n$ denotes, as usual, the projection of $\mathbf{x}$ on the $n$-th coordinate. 
Following the usual practice, 
we say that \emph{$\mu_p$ has decay of correlations} if $\lim_{n\to\infty}|C_p(n)|=0$.

\medskip\noindent 
Let $p^*:=\sup\{p\in(0,1):\ C_p(n)>0\ \forall n\in\mathbb{N}\}$, and for each 
$p\in (0,1)$, let 
\begin{equation}\label{eq:scalingexponent}
\beta_p:=\frac{\log(2-p)-\log(1-2p)-\log(2-3p)}{\log(2-p)}.
\end{equation}

\medskip \noindent
\begin{theorem}[Power-law Decay of Correlations]\label{theo:asymptoticscaling}
For each $p\in (0,p^*)$ there exist $n_0\in\mathbb{N}$ and constants $A_p\leq 1\leq B_p$ 
such that 
\[
A_p\, n^{-\beta_p}\leq C_p(n)\leq B_p\, n^{-\beta_p}
\] 
for all $n\geq n_0$.
\end{theorem}

\medskip
\subsection{} 
The invariance of the expansion-modification dynamics under coordinatewise negation, or 
\emph{flip invariance}, $0\leftrightarrow 1$, implies that 
$\mu_p\{{x}_n = 1\}=\mu_p\{{x}_n=0\}=1/2$ for all $n\in\mathbb{N}_0$. Therefore
\[
C_p(n):=\int_{X} {x}_0 {x}_n\,d\mu_p(\mathbf{x})-1/4\equiv
\mu_p\{ {x}_0={x}_n=1\}-1/4.
\]
Now, flip invariance also implies $\mu_p\{ {x}_0 = {x}_n=1\}=\mu_p\{{x}_0={x}_n=0\}$ and 
$\mu_p\{{x}_0=0\neq {x}_n=1\}=\mu_p\{{x}_0=1\neq {x}_n=0\}$, for each $n\in\mathbb{N}_0$. 
With this, we obtain a very simple expression for the two-sites correlation 
\[
C_p(n):=\frac{1}{2}\left(\mu_p\{{x}_0={x}_n\}-1/2\right)=
\frac{1}{4}\left(\mu_p\{{x}_0={x}_n\}-\mu_p\{{x}_0\neq{x}_n\}\right).
\]
This expression and the invariance under the expansion-modification dynamics will allow 
to deduce a recurrence formula for the two-sites correlation function.

\medskip\noindent Let us denote by $\ell(\mathbf{s}_0^k)$ the length of the words obtained 
by applying the substitution $\mathbf{s}_0^k$, and for each $k, n\in \mathbb{N}$ let
\[\nu_p(k,n)   :=   \nu_p\{\ell(\mathbf{s}_0^{k-2})=n-1\} 
            \equiv \binom{k-1}{n-k}(1-p)^{n-k}p^{2k-n-1}.  \]
Since $\mu_p$ is left invariant under the expansion-modification dynamics, and since 
$\nu_p$ is a Bernoulli measure, then
\begin{eqnarray*}
\mu_p\{{x}_0={x}_n\}&=&\sum_{k=\lceil n/2\rceil}^n
\mu_p\{{x}_0={x}_k\}\left(p^2\nu_p(k,n)+(1-p)^2(\nu_p(k,n-1)+\nu_p(k,n-2)\right)\\
                    & &+\sum_{k=\lceil n/2\rceil}^n
\mu_p\{{x}_0\neq{x}_k\}\left(2p(1-p)\nu_p(k,n-1)+p(1-p)\nu_p(k,n)\right),
\end{eqnarray*}
and similarly for $\mu_p\{{x}_0\not={x}_n\}$. From the previous equation and its analogous 
for $\mu_p\{{x}_0\not={x}_n\}$, it readily follows that
\begin{equation}\label{eq:reccorrelation}
C_p(n)=\sum_{k=\lfloor n/2\rfloor}^nC_p(k) \left(f(p)\,\nu_p(k,n)+
                         g(p)\,\nu_p(k,n-1)+h(p)\,\nu_p(k,n-2)\right),
\end{equation}
with $f(p):=p(2p-1)$, $g(p):=(1-p)(1-3p)$, and $h(p):=(1-p)^2$.
This relation can be rewritten as 
\begin{equation}\label{eq:explicitrec}
C_p(n)=\frac{1}{1+p^n(1-2p)}\sum_{k=[n/2]}^{n-1}C_p(k)\mathcal{W}_p(k,n),
\end{equation}
with $\mathcal{W}_p(k,n)$ defined as
\begin{equation}\label{eq:Wpdekyn}
\mathcal{W}_p(k,n) := f(p)\nu_p(k,n) + g(p)\nu_p(k,n-1) + h(p)\nu_p(k,n-2). 
\end{equation}

\medskip \noindent We will make extensive use of the previous relations in the proof 
of Theorem~\ref{theo:asymptoticscaling}.

\bigskip
\section{The scaling exponent}\label{sec:heuristic}

\subsection{}According to Theorem~\ref{theo:asymptoticscaling}, the correlation 
function follows an asymptotic scaling law for mutation probabilities in the 
range $0< p <p^*$. It also establishes an expression for the scaling exponent $\beta_p$. 
The following argument leads us to conjecture that the scaling property, and the 
expression for the corresponding scaling exponent, extend to the whole interval 
$0 < p < 1$. 

\medskip\noindent
Let us consider the recursive relation
\[
C_p(n)=\sum_{k=\lfloor n/2\rfloor}^n
             C_p(k)\left(f(p)\,\nu_p(k,n)+g(p)\,\nu_p(k,n-1)+h(p)\,\nu_p(k,n-2)\right),
\]
deduce above. The distribution $k\mapsto \nu_p(k,n)$, which is unimodal with maximum at 
$k\approx n/(2-p)$, steepens around this maximum as $n$ goes to infinity in such a way 
that
\[\sum_{k=\lfloor n/2\rfloor}^n \nu_p(k,n)\approx 
                          \sum_{\ell(n)\leq k\leq u(n)} \nu_p(k,n),\]
where $\ell(n) < n/(2-p)< u(n)$, which we define below in Subsection~\ref{sub:scaling}, 
are such that both $(2-p)\ell(n)/n$ and $(2-p)u(n)/n$ tend to 1 as $n$ goes to infinity. 
Hence, assuming a slow variation in $k\mapsto C_p(k)$, we have
\begin{eqnarray*}
C_p(n)&\approx&\sum_{\ell(n)\leq k\leq u(n)}
C_p(k)\left(f(p)\,\nu_p(k,n)+g(p)\,\nu_p(k,n-1)+h(p)\,\nu_p(k,n-2)\right)\\
       &\approx& C\left(\frac{n}{2-p}\right)\sum_{\ell(n)\leq k\leq u(n)}
              \left(f(p)\,\nu_p(k,n)+g(p)\,\nu_p(k,n-1)+h(p)\,\nu_p(k,n-2)\right)\\
       &\approx& C\left(\frac{n}{2-p}\right)\sum_{k=\lfloor n/2\rfloor}^n
              \left(f(p)\,\nu_p(k,n)+g(p)\,\nu_p(k,n-1)+h(p)\,\nu_p(k,n-2)\right)\\ 
       &   =   & C\left(\frac{n}{2-p}\right)\left(f(p)\,S_p(n)+
                                             g(p)\,S_p(n-1)+h(p)\,S_p(n-2)\right),                  
\end{eqnarray*}
where $S_n(p):=\sum_{k=\lfloor n/2\rfloor}^n\nu_p(k,n)=(1-(p-1)^n)/(2-p)$. 
From this we finally obtain the approximate scaling relation
\[
C_p((2-p)^k\,n_0)\approx \left(\frac{(1-2p)(2-3p)}{(2-p)}\right)^kC(n_0),
\]
which traduces into the scaling law $C_p(n)\approx C_p(n_0)\, (n/n_0)^{-\beta_p}$, with 
\[\beta_p:=\frac{\log(2-p)-\log(1-2p)-\log(2-3p)}{\log(2-p)}.\]

\medskip \subsection{} We have used the recurrence relation in 
Equation~\eqref{eq:explicitrec} to numerically compute the two-sites correlation 
function for different values of the mutation probability. As shown in 
Figure~\ref{fig:powerlaws}, the numerical computations confirm that the two-point 
correlation function approximatively 
follows a power-law behavior. 
Furthermore, according to Figure~\ref{fig:comparisonbetas}, the theoretically predicted
exponents, $\{-\beta_p:\ 0 < p < 1\}$, fit very well the ones obtained by linear 
regression from the numerically computed correlation functions. 

\begin{center}
\begin{figure}[h]
\includegraphics[scale=0.45]{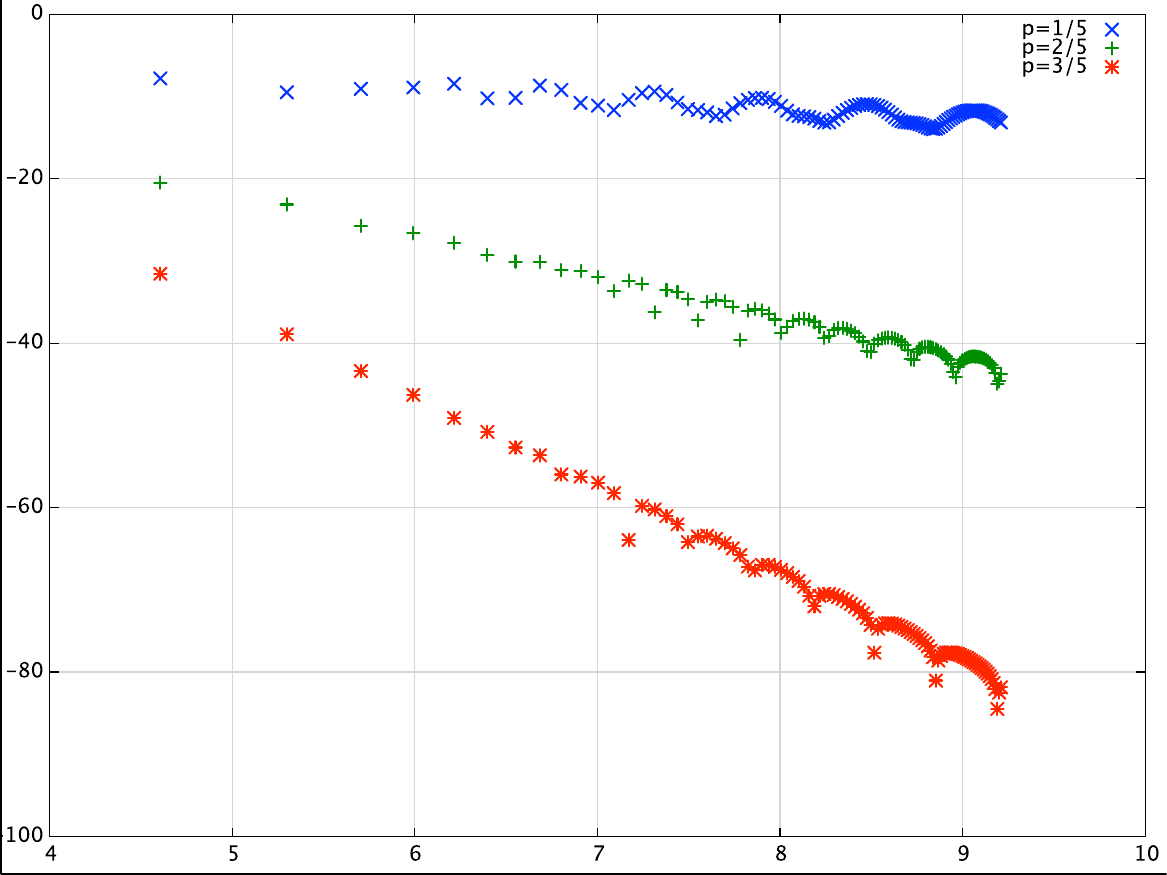}
\caption{Log-log plot of the two-sites correlation function. A power-law 
behavior clearly appears.}\label{fig:powerlaws}
\end{figure}
\end{center}

\begin{center}
\begin{figure}[h]
\includegraphics[scale=0.45]{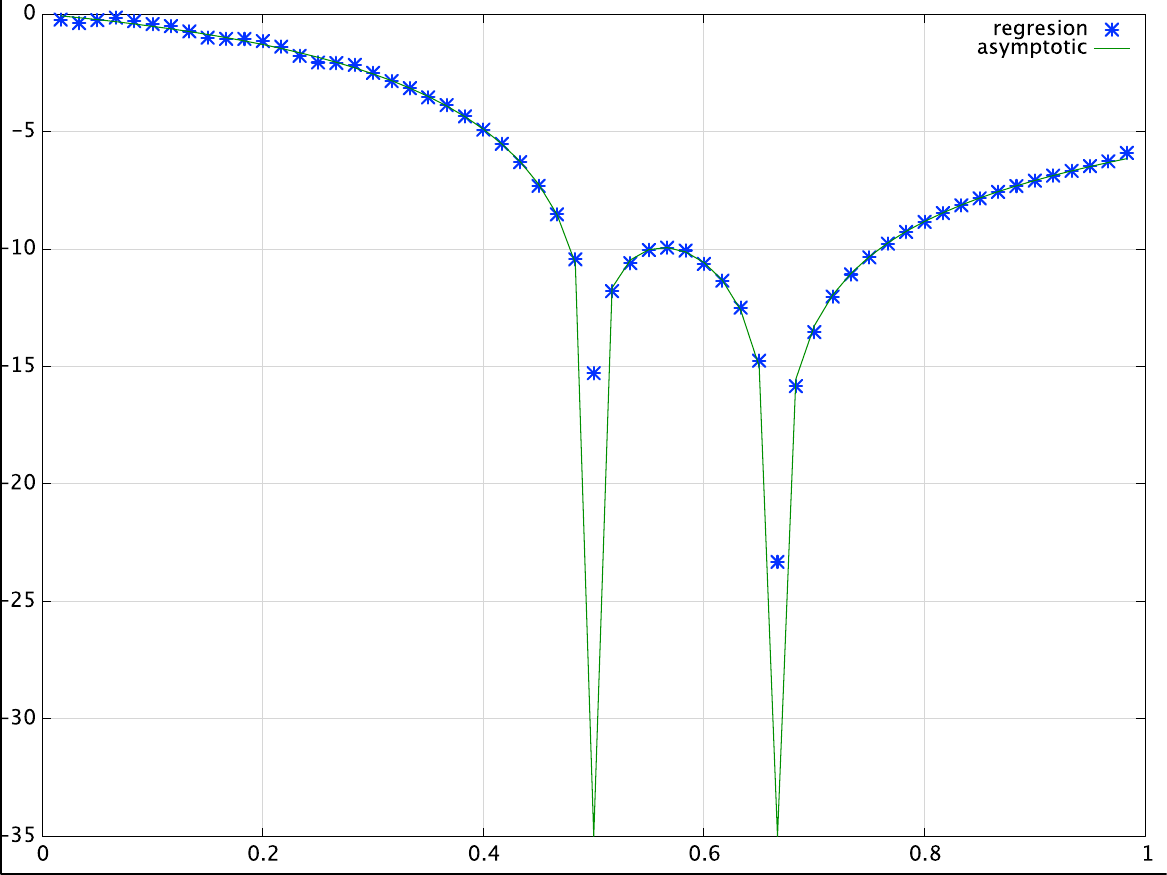}
\caption{Exponents obtained by the best power-law fit to the two-sites correlation 
function compared to the theoretical asymptotic exponents 
$-\beta_p$.}\label{fig:comparisonbetas}
\end{figure}
\end{center}

\medskip\noindent
The argument developed above suggests that the stationary measure $\mu_p$ varies in a
piecewise smooth way with $p$. This variation is reflected on the behavior of the two-sites 
correlation function $C_p$, which appears to follow a power law decay which prevails in the 
whole interval $0 < p < 1$, except for the two singularities located at $p=1/2$ and $p=2/3$.
At precisely those values of $p$, the two-sites correlation function appears to decay faster 
than any power law.


\section{Proofs}\label{sec:proofs}

\subsection{}
Our proof of Theorem~\ref{theo:asymptotic} requires the following.

\medskip
\begin{lemma}\label{lem:irreducible}
For each $\mathbf{a}, \mathbf{b}\in\{0,1\}^{\ell+1}$ there exists 
$\{\mathbf{s}^0,\mathbf{s}^1,\ldots, \mathbf{s}^n\}\subset\{\mathrm{e},\mathrm{m}\}^+$, 
such that $\mathbf{b}\sqsubseteq
\mathbf{s}^n\circ\cdots\circ\mathbf{s}^{1}\circ\mathbf{s}^0(\mathbf{a})$.
\end{lemma}

\medskip
\begin{proof}
For $\ell\in\mathbb{N}_0$ and $\mathbf{a}\in\{0,1\}^{\ell+1}$, let 
$\mathbf{s}=\mathrm{e}^{\ell+1}$ whenever $a_0=0$, otherwise let 
$\mathbf{s}=\mathrm{m}\,\mathrm{e}^{\ell}$.
Clearly, for each  $n >\lceil\log(\ell+1)/\log(2)\rceil$ we have
\[
0^{\ell+1}\sqsubseteq \mathbf{t}^n\circ\cdots\circ\mathbf{t}^1\circ\mathbf{s}(\mathbf{a}),
\]
where $\mathbf{t}^j\in\{\mathrm{e},\mathrm{m}\}^+$ is such that 
$\mathrm{e}^{\ell+1}\sqsubseteq\mathbf{t}^j$ for each $1\leq j\leq n$. 
We claim that for each
$\mathbf{b}\in\{0,1\}^{\ell+1}$ there exists 
$\{\mathbf{s}^0,\mathbf{s}^1,\ldots, \mathbf{s}^k\}\subset\{\mathrm{e},\mathrm{m}\}^+$ such 
that
\[
\mathbf{b}\sqsubseteq
\mathbf{s}^k\circ\cdots\circ\mathbf{s}^{1}\circ\mathbf{s}^0\left(0^{\ell+1}\right),
\]
which readily implies the result.

\noindent
Since $0\sqsubseteq\mathrm{e}(0)$ and $1\sqsubseteq\mathrm{m}(0)$, the claim holds for 
$\ell=0$. Assuming the claim for $\ell=l-1$ we have, for all $\mathbf{c}\in\{0,1\}^{l+1}$, 
a sequence  
$\{\mathbf{s}^0,\mathbf{s}^1,\ldots, \mathbf{s}^k\}\subset\{\mathrm{e},\mathrm{m}\}^+$ 
such that
$\mathbf{c}_1^{l}\sqsubseteq
\mathbf{s}^k\circ\cdots\circ\mathbf{s}^{1}\circ\mathbf{s}^0\left(0^{l}\right)$.
If $k$ even and $c_0=1$ or $k$ odd and $c_0=0$, by taking 
$\mathbf{t}^j:=\mathrm{m}\,\mathbf{s}^j$ for $0\leq j\leq k$, we have 
\[
\mathbf{c}=c_0\,\mathbf{c}_1^{m}\sqsubseteq
           \mathrm{m}^{k+1}(0)\,
           \mathbf{s}^k\circ\cdots\circ\mathbf{s}^{1}\circ\mathbf{s}^0\left(0^{l+1}\right)
           =\mathbf{t}^k\circ\cdots\circ\mathbf{t}^{1}\circ\mathbf{t}^0\left(0^{l+1}\right).
\]
On the other hand, if $k$ even and $c_0=0$ or $k$ odd and $c_0=1$, by taking 
$\mathbf{t}^{j+1}:=\mathrm{m}\,\mathbf{s}^j$ for $0\leq j\leq k$, and 
$\mathbf{t}^{0}:=\mathrm{e}^{l+1}$, we have
\[
\mathbf{c}=c_0\,\mathbf{c}_1^{m}\sqsubseteq
      \mathrm{m}^{k+2}(0)\,
      \mathbf{s}^k\circ\cdots\circ\mathbf{s}^{1}\circ\mathbf{s}^0\left(0^{l+1}\right)
      =\mathbf{t}^{k+1}\circ\cdots\circ\mathbf{t}^{1}\circ\mathbf{t}^0\left(0^{l+1}\right),
\]
and the lemma follows.
\end{proof}

\medskip

\noindent{\bf{\em Proof of Theorem~\ref{theo:asymptotic}}.} Let us assume that for each 
$\ell\in{\mathbb N}_0$, the stochastic matrix $M_\ell$ is primitive. This implies, by 
the Perron-Frobenius Theorem, that there is a unique probability vector  
$\mu_\ell:\{0,1\}^{\ell+1}\to (0,1]$ such that 
\[
\mu_\ell=\mu_\ell M_\ell \hskip 10pt \text{ and } \hskip 10pt 
\lim_{t\to\infty}\mu_\ell^0M_\ell^t=\mu_\ell,
\]
for every initial probability vector $\mu^0_\ell:\{0,1\}^{\ell+1}\to [0,1]$. Hence,
for any measure $\mu^0$ specifying the distribution of the initial conditions, for each
$\ell\in\mathbb{N}_0$, and for all $\mathbf{a}\in\{0,1\}^{\ell+1}$ we have
\begin{equation}\label{eq:convergence}
\lim_{t\to\infty}\mu^t[\mathbf{a}]=\mu_\ell(\mathbf{a}).
\end{equation}
If in addition the probability vectors $\mu_\ell$ satisfy the compatibility
condition 
\begin{equation}\label{eq:compatibility}
\sum_{x\in\{0,1\}}\mu_{\ell+1}(\mathbf{a}x)=\mu_{\ell}(\mathbf{a}),
\end{equation}
for each $\ell\in\mathbb{N}_0$ and $\mathbf{a}\in\{0,1\}^{\ell+1}$, then Kolmogorov's 
representation theorem implies the existence of a measure $\mu$ on $X$ such 
that $\mu[\mathbf{a}]=\mu_{\ell}(\mathbf{a})$ for 
each $\ell\in\mathbb{N}_0$ and $\mathbf{a}\in\{0,1\}^{\ell+1}$. 
Finally, Equation~\eqref{eq:convergence} ensures the convergence of $\mu^t$ 
towards $\mu$ in the *-weak sense.

\medskip \noindent
The primitivity of $M_\ell$ follows straightforwardly from the following argument. As 
proved in Lemma~\ref{lem:irreducible}, for each pair of words 
$\mathbf{a}, \mathbf{b}\in\{0,1\}^{\ell+1}$, there exist a sequence 
of substitutions such that applied to $\mathbf{a}$ produces a word having $\mathbf{b}$ 
as prefix. 
Now, since all words in $\{\mathrm{e},\mathrm{m}\}^{\ell+1}$ have positive probability, 
then the previous claim implies that $M_\ell^n(\mathbf{a},\mathbf{b})>0$ for some $n>0$, 
which proves that $M_\ell^n$ is irreducible. Now, since the word $00\cdots 0$ occurs as 
the prefix of 
$\mathrm{e}(0)\mathrm{e}(0)\cdots \mathrm{e}(0)$, then $M_\ell(00\cdots 0,00\cdots 0)>0$, 
which implies that $M_\ell$ is aperiodic, and so $M_\ell$ is primitive.

\medskip \noindent
To prove the compatibility condition~\eqref{eq:compatibility}, we should notice that it 
is inherited from the analogous compatibility condition satisfied by all the marginals 
$\mu_\ell^t$ at each time $t\in\mathbb{N}_0$. Indeed, for $t=0$ we obviously have
\[
\sum_{x\in\{0,1\}}\mu^0_{\ell+1}(\mathbf{a}x):=\sum_{x\in\{0,1\}}\mu^0[\mathbf{a}x]=
\mu^0\left(\bigsqcup_{x\in\{0,1\}} [\mathbf{a}x]\right)=\mu[\mathbf{a}]=:
\mu^0_{\ell}(\mathbf{a}),
\]
for each $\ell\in\mathbb{N}_0$ and $\mathbf{a}\in\{0,1\}^{\ell+1}$. Here $\sqcup$ stands 
for the disjoint union. Now, from Equation~\eqref{eq:marginals} it follows that
\begin{eqnarray*}
\sum_{x\in\{0,1\}}\mu^{t+1}_{\ell+1}(\mathbf{a}x)
&=&
\sum_{x\in\{0,1\}}\sum_{\mathbf{s}\in \{\mathrm{e},\mathrm{m}\}^{\ell+2}}
\nu_p[\mathbf{s}]\left(
\sum_{\mathbf{b}\in \{0,1\}^{\ell+2}\atop \mathbf{a}x \sqsubseteq
\prod_{i=0}^{\ell+1} s_i(b_i)} \mu^{t}[\mathbf{b}]\right)\\
&=&
\sum_{\mathbf{s}\in \{\mathrm{e},\mathrm{m}\}^{\ell+2}}\nu_p[\mathbf{s}]\, 
\mu^{t}\left(
\bigsqcup_{x\in\{0,1\}}\bigsqcup_{\mathbf{b}\in \{0,1\}^{\ell+2}\atop \mathbf{a}x 
\sqsubseteq
\prod_{i=0}^{\ell+1} s_i(b_i)} [\mathbf{b}]\right)\\
&=&
\sum_{\mathbf{s}\in \{\mathrm{e},\mathrm{m}\}^{\ell+2}}\nu_p[\mathbf{s}]\, \mu^{t}\left(
\bigsqcup_{\mathbf{b}\in \{0,1\}^{\ell+2}\atop \mathbf{a} \sqsubseteq
\prod_{i=0}^{\ell+1} s_i(b_i)} [\mathbf{b}]\right).
\end{eqnarray*}
Since $|\mathbf{a}|=\ell+1$, the statement 
$\mathbf{a}\sqsubset \prod_{i=0}^{\ell+1} s_i(b_i)$ is equivalent to 
$\mathbf{a}\sqsubset \prod_{i=0}^{\ell} s_i(b_i)$, and we have
\begin{eqnarray*}
\sum_{x\in\{0,1\}}\mu^{t+1}_{\ell+1}(\mathbf{a}x)
&=&
\sum_{\mathbf{s}\in \{\mathrm{e},\mathrm{m}\}^{\ell+2}}\nu_p[\mathbf{s}]\, \mu^{t}\left(
\bigsqcup_{\mathbf{b}\in \{0,1\}^{\ell}\atop \mathbf{a} \sqsubseteq
\prod_{i=0}^{\ell} s_i(b_i)} [\mathbf{b}]\right)\\
&=&
\sum_{\mathbf{s}\in \{\mathrm{e},\mathrm{m}\}^{\ell+1}} \left(
\sum_{\mathbf{b}\in \{0,1\}^{\ell+1}\atop \mathbf{a} \sqsubseteq
\prod_{i=0}^{\ell} s_i(b_i)} \mu^{t}[\mathbf{b}]\right)\sum_{\rho\in\{\mathrm{e},\mathrm{m}\}} 
\nu_p[\mathbf{s}\rho]\\
&=&
\sum_{\mathbf{s}\in \{\mathrm{e},\mathrm{m}\}^{\ell}} \left(
\sum_{\mathbf{b}\in \{0,1\}^{\ell+1}\atop \mathbf{a} \sqsubseteq
\prod_{i=0}^{\ell} s_i(b_i)} \mu^{t}[\mathbf{b}]\right)
=\left(\mu_\ell^t\,M_\ell\right)(\mathbf{a}):=\mu_\ell^{t+1}(\mathbf{a})
\end{eqnarray*}
for each $\ell\in\mathbb{N}_0$ and $\mathbf{a}\in\{0,1\}^{\ell+1}$. The compatibility
condition~\eqref{eq:compatibility} follows by taking the limit
$t\to\infty$ on both sides of the equation, which completes the proof of the theorem.
\rightline{$\Box$}

\medskip\subsection{}\label{sub:scaling}  
This subsection is devoted to the proof of Theorem~\ref{theo:asymptoticscaling}, which 
relies in Lemmas~\ref{theo:upperbound} and~\ref{theo:lowerbound} and other auxiliary 
propositions which we state without proof. The proofs of these auxiliary results is 
referred to Subsections~\ref{subsec:upper}, \ref{subsec:auxiliars} and \ref{subsec:lower}.

\medskip
\begin{lemma}[Upper bound]\label{theo:upperbound}
For each $p\in (0,1)$, the stationary measure $\mu_p$ has decay of correlations bounded above 
by a power law. Indeed, for each $p\in (0,1)$ there exist positive constants $\alpha_p$ and 
$K_p >0$ such that
\[
|C_p(n)|\leq K_p n^{-\alpha_p}
\]
for all $n\geq 2$. 
\end{lemma}

\medskip
\begin{lemma}[Lower bound]\label{theo:lowerbound}
There are constants $0 < b$ and $n_{p^*}\in\mathbb{N}$, such that $C_p(n)\geq n^{-b}$ for 
all $p\in (0,p^*)$ and all $n\geq n_{p^*}$. 
\end{lemma}

\medskip\noindent This lemma has the following straightforward consequences that we will 
use below.
 
\medskip
\begin{corollary}\label{cor:upperbound}
For each $p\in (0,1)$ and  $a < \alpha_p$, there exists $n=n_a$ such that 
\[
C_p(n)\leq n^{-a}
\]
for each $n\geq n_a$.
\end{corollary}

\medskip \noindent
Besides Lemmas~\ref{theo:upperbound} and~\ref{theo:lowerbound}, the proof of 
Theorem~\ref{theo:asymptoticscaling} requires 
some additional notation and preliminary results which we present now.

\medskip \noindent Let $b$ be as in Theorem~\ref{theo:lowerbound}. Then for each 
$\beta>b$ define the function $d_{\beta}:[1,\infty) \to \mathbb{R}$ as follows
\[ d_{\beta}(x):=\sqrt{p(1-p)(2-p)(\beta+1)\log(x)/x} \] 
With this, define the functions $\ell_{\beta,p},u_{\beta,p}: [1,\infty)\to [0,\infty)$ 
given by
\begin{equation*}
\ell_{\beta,p}(x):=\frac{x}{2-p+d_\beta(x)},\hskip 10pt 
                  u_{\beta,p}(x):=\frac{x}{2-p-d_\beta(x)}.
\end{equation*}
Finally, to simplify the expressions that will appear, define
\[ \lambda_p:=(2-p),  \hskip 10pt 
     \phi_{\beta,p}(x):=\ell_{\beta,p}(\lambda_p\,x)/\lambda_p \, \text{ and }\, 
\psi_{\beta,p}(x):=u_{\beta,p}(\lambda_p\,x)/\lambda_p.\]

\medskip \noindent Let us remind that, for $p\in (0,1)$ and each $k, n\in \mathbb{N}$,  
\[\mathcal{W}_p(k,n) := f(p)\nu_p(k,n) + g(p)\nu_p(k,n-1) + h(p)\nu_p(k,n-2),\]
with $f(p):=p(2p-1)$, $g(p):=(1-p)(1-3p)$, and $h(p):=(1-p)^2$ and 
\[\nu_p(k,n):=\nu_p\{\ell(\mathbf{s}_0^{k-2})=n-1\} \equiv 
                             \binom{k-1}{n-k}(1-p)^{n-k}p^{2k-n-1}.\] 
We have the following. 

\newpage
\begin{proposition}\label{lem:Wp>0}
If $p \in (0, p^*)$ and $\lfloor n/2\rfloor < k\leq 2(n-1)/3$, then $\mathcal{W}_p(k,n)>0$.
\end{proposition}

\medskip
\begin{proof}
A simple computation shows that
\begin{eqnarray*}
\mathcal{W}_p(k,n)& =&\frac{(1-p)^{n-k}p^{2k-n}(k-1)!}{(n-k)!(2k-n+1)!}Q_p(n,k), 
                                                \text{ with }\\
          Q_p(k,n)&:=& \left((1-2p)(2n-3k)(2k-n+1)+ p(2n-3k-2)(n-k)\right).
\end{eqnarray*}
Since $p^*<1/2$, then $Q_p(k,n)>0$ for $\lfloor n/2\rfloor < k\leq 2(n-1)/3$. The result 
follows from the fact that $\mathcal{W}_p(k,n)$ and $Q_p(k,n)$ have the same sign.
\end{proof}

\medskip
\begin{proposition}\label{prop:sumnuptail}
For each $\beta>b$ there exists $n_\beta \geq 5$ such that
\[\delta_n:=n^{b}\, \left(\sum_{|n/k-(2-p)|>d_\beta(n)}\nu_p(k,n)\right) 
\leq n^{-(\beta-b)/2}, \]
for all $n\geq n_\beta$. 
\end{proposition}

\medskip
\begin{proposition}\label{prop:correlationtail}
For each $n\geq 5$ and $p \in (0, p^*)$ we have 
\[
\left|\sum_{k > 2(n-1)/3}\mathcal{W}_p(k,n)C_p(k)\right|  \leq 
\frac{n\left(4p(1-p)\right)^{(n-2)/3}}{6}.
\]
\end{proposition}

\medskip 
\begin{proposition}\label{lem:QR}
There exists $x_0\geq e$ such that, for each $x\geq x_0$ there are constants 
$0 < Q(x) < 1 < R(x)$ such that for every $1\leq j\leq k\in\mathbb{N}$ we have
\[ 
Q(x)\,\lambda_p^{k-j}\,x\leq\lambda_p\phi_{\beta,p}^j(\lambda^{k-1}\,x)\leq 
                       u_{\beta,p}^j(\lambda_p^{k-1}\,x)\leq R(x)\,\lambda_p^{k-j}\,x.
\]  
Furthermore, $Q(x),R(x)\rightarrow 1$ as $x\to\infty$. 
\end{proposition}

\medskip\noindent 
{\bf{\em Proof of Theorem~\ref{theo:asymptoticscaling}}.}
Fix $p\in(0,p^*)$ and $\beta > b$, and let $d=d_\beta$, $u=u_{\beta,p}$ and 
$\ell=\ell_{\beta,p}$ be as above. 
Let $f,g,h:[0,1]\to\mathbb{R}$ be as in Equation~\eqref{eq:Wpdekyn}.
Note that $\ell(n)\leq k\leq u(n)$ implies $\left|\frac{n}{k}-(2-p)\right|\leq d(n)$. 
Note also that $|C_p(k)|\leq 1/4$ for all $k$, and that $|f(p)|+|g(p)|+ |h(p)| \leq 2$ 
for all $p\in (0,1)$. From this, using Equation~\eqref{eq:reccorrelation} and  
Proposition~\ref{prop:sumnuptail}, we obtain
\begin{eqnarray}
C_p(n)&\leq& \sum_{\ell(n)\leq k\leq u(n)}C_p(k)\mathcal{W}_p(k,n)
+n^{-b}\, \delta_n \label{eq:upineqcorrelation}\\
C_p(n)&\geq& \sum_{\ell(n)\leq k\leq u(n)} C_p(k)\mathcal{W}_p(k,n)
-n^{-b}\, \delta_n \label{eq:lowineqcorrelation}
\end{eqnarray} 
for each $n\geq n_\beta$.

\medskip\noindent A simple computation shows that for $p<p^*$, there exists 
$n_1=n_1(p,\beta)\in\mathbb{N}$ such that $u(n) < 2(n-1)/3$ for all $n\geq n_1$. Hence,
according to Proposition~\ref{lem:Wp>0},  $\mathcal{W}_p(k,n)>0$  for all $n\geq n_1$ 
and $\ell(n)\leq k\leq u(n)$. In this case we can define a probability distribution 
$k\mapsto \mathbb{P}_p(k,n)$ proportional to $k\mapsto \mathcal{W}_p(k,n)$, in the 
interval $\ell(n)\leq k\leq u(n)$. 

\medskip \noindent For $n\geq\max(n_1,n_{p^*},n_\beta)$ we can use the lower bound of 
Theorem~\ref{theo:lowerbound}, and rewrite Inequalities~\eqref{eq:upineqcorrelation} 
and~\eqref{eq:lowineqcorrelation} as 
\begin{eqnarray*}
C_p(n)&\leq& \left(\sum_{k=\lfloor n/2\rfloor}^n\mathcal{W}_p(n,k)+2\delta_n\right)
\mathbb{E}_{p,n}(C_p), \\
C_p(n)&\geq& \left(\sum_{k=\lfloor n/2\rfloor}^n\mathcal{W}_p(n,k)-2\delta_n\right)
\mathbb{E}_{p,n}(C_p),
\end{eqnarray*}  
where $\mathbb{E}_{p,n}(C_p)$ denotes the mean value of $C_p$ with respect to 
$\mathbb{P}_p(k,n)$. Using the fact that $\sum_{k=\lfloor n/2\rfloor}^n 
\mathcal{W}_p(k,n)=(1-2p)(2-3p)/(2-p)-2p(p-1)^n$, we obtain
\begin{eqnarray*}
C_p(n)&\leq& \left(\frac{(1-2p)(2-3p)}{2-p}+3\delta_n\right)\mathbb{E}_{p,n}(C_p)\\
      &\leq& \left(\frac{(1-2p)(2-3p)}{2-p}+3\delta_n\right)\max_{\ell(n) 
      \leq k\leq u(n)}C_p(k),\\
C_p(n)&\geq& \left(\frac{(1-2p)(2-3p)}{2-p}-3\delta_n\right)\mathbb{E}_{p,n}(C_p)\\
      &\geq& \left(\frac{(1-2p)(2-3p)}{2-p}+3\delta_n\right)\min_{\ell(n) 
      \leq k\leq u(n)}C_p(k).
\end{eqnarray*}
We can rewrite these inequalities as
\[
\lambda_p^{-\beta_p-\eta_{\lambda_p x}}\min_{\lambda_p\phi(x)
          \leq y\leq \lambda_p\psi(x)}C_p([y])
          \leq C_p([\lambda_p\,x])
          \leq \lambda_p^{-\beta_p+\eta_{\lambda_p x}}\max_{\lambda_p\phi(x)
          \leq y\leq \lambda_p\psi(x)}C_p([y]),
\]
with $\eta_x:=3\lambda_p^{-\beta_p}\,(x-1)^{-(\beta-b)/2}/\log(\lambda_p)
           \geq 3\lambda_p^{-\beta_p}\,\delta_{[x]}/\log(\lambda_p)$ 
and $\beta_p$ as defined in Equation~\eqref{eq:scalingexponent}. It follows from here, 
by a straightforward induction, that
\begin{eqnarray*}
C_p([\lambda_p^k\,x])&\leq& 
               m_p^{-k\,\beta_p+\sum_{j=0}^{k-1}\eta_{\lambda_p\phi^j(\lambda_p^{k-1}x)}}
                           \max_{\lambda_p\phi^k(\lambda_p^{k-1}x)
                      \leq z\leq \lambda_p\psi^2(\lambda_p^{k-1}x)}C_p([z]),\\
C_p([\lambda_p^k\,x])&\geq& m_p^{-k\,\beta_p-\sum_{j=0}^{k-1}
                            \eta_{\lambda_p\phi^j(\lambda_p^{k-1}x)}}
                            \min_{\lambda_p\phi^k(\lambda_p^{k-1}x)\leq z
                       \leq \lambda_p\psi^2(\lambda_p^{k-1}x)}C_p([z]),                        
\end{eqnarray*}
for all $k\in\mathbb{N}$ and $x\geq \max(n_1,n_{p^*},n_\beta)$. 

\medskip \noindent Using the above estimates and taking into account 
Corollary~\ref{cor:upperbound} and Theorem~\ref{theo:lowerbound}, we obtain, for all 
$p\in (0,p^*)$ and all $x\geq n_0:=\max(n_1,n_{p^*},n_\beta,n_\alpha)$, 
the inequalities
\begin{eqnarray*}
C_p([\lambda_p^k\,x])&\leq& 
               \lambda_p^{-k\,\beta_p+\sum_{j=0}^{k-1}\eta_{Q(x)\,\lambda_p^{k-j}\,x} }
               \left(Q(x)\,x\right)^{-a}
                       \leq            
               \lambda_p^{-k\,\beta_p+\epsilon_p(x)}\, \left(Q(x)\,x\right)^{-a}, \\
C_p([\lambda_p^k\,x])&\geq& 
               \lambda_p^{-k\,\beta_p-\sum_{j=0}^{k-1}\eta_{Q(x)\,\lambda_p^{k-j}\,x}}
               \left(R(x)\,x\right)^{-b}  
                       \geq 
               \lambda_p^{-k\,\beta_p-\epsilon_p(x)}\,\left(R(x)\,x\right)^{-b},                       
\end{eqnarray*}
where
\[
\epsilon_p(x):=\frac{3\times\left(2\,\lambda_p\right)^{(\beta-b)/2}
                    }{\left(Q(x)\,x\right)^{(\beta-b)/2}\lambda_p^{\beta_p}\,\log(\lambda_p)
\left(\lambda_p^{(\beta-b)/2}-1\right)}\rightarrow 0 \text{ as } x\to \infty.
\]
For $n\geq n_0$ let $k\in\mathbb{N}$ be such that $n_0\lambda_p^k\leq n <n_0\lambda_p^{k+1}$,
and let $x_0:=n/\lambda_p^k\in [n_0,\lambda_p\,n_0)$. With this, the inequalities above 
become
\[
R(x_0)^{-b}\,x_0^{\beta_p-b}\lambda_p^{-\epsilon_p(x_0)}\,n^{-\beta_p} \leq  C_p(n) 
\leq
Q(x_0)^{-a}\,x_0^{\beta_p-a}\lambda_p^{-\epsilon_p(x_0)}\, n^{-\beta_p},
\]
and the result follows by taking
\begin{eqnarray*}
A_p&:=&\min_{n_0\leq x_0< \lambda_p\,n_0} 
       R(x_0)^{-b}\,x_0^{\beta_p-b}\lambda_p^{-\epsilon_p(x_0)},\\
B_p&:=&\max_{n_0\leq x_0< \lambda_p\,n_0} 
       Q(x_0)^{-a}\,x_0^{\beta_p-a}\lambda_p^{-\epsilon_p(x_0)}.
\end{eqnarray*}
\rightline{$\Box$}

\subsection{}\label{subsec:upper}
\medskip\noindent {\bf{\em Proof of Lemma~\ref{theo:upperbound}}.} For each 
$n\in\mathbb{N}$ let $\bar{C}_p(n):=\max\{|C_p(m)|:\ m\geq n\}$.
Obviously $n\mapsto\bar{C}_p(n)$ is a non-increasing upper bound for $n\mapsto|C_p(n)|$. 

\medskip \noindent For each $n\in\mathbb{N}$ let 
$S_p(n):= \sum_{k=\lfloor n/2\rfloor}^n \nu_p(k,n)$. 
From Equation~\eqref{eq:reccorrelation} we readily derive the inequality
\begin{eqnarray*}
|C_p(n)|&\leq& \sum_{k=\lfloor n/2\rfloor}^n |C_p(k)| 
             \left(|f(p)|\,\nu_p(k,n)+|g(p)|\,\nu_p(k,n-1)+|h(p)|\,\nu_p(k,n-2)\right)\\
       &\leq& \bar{C}_p(\lfloor n/2\rfloor)
              \left(|f(p)|S_p(n)+|g(p)|S_p(n-1)+ |h(p)|S_p(n-2)\right),
\end{eqnarray*}
which holds for each $n\geq 2$. Hence,
\begin{eqnarray*}
\bar{C}_p(n)&\leq &\max_{m\geq n} \bar{C}_p(\lfloor m/2\rfloor)
                   \left(|f(p)|S_p(m)+|g(p)|S_p(m-1)+ |h(p)|S_p(m-2)\right)\\
            &\leq& \bar{C}_p(\lfloor n/2\rfloor)\max_{m\geq n}
                   \left(|f(p)|S_p(m)+|g(p)|S_p(m-1)+ |h(p)|S_p(m-2)\right) 
\end{eqnarray*}
It can be easily verified that $S_p(n+1)=pS_p(n)+(1-p)S_p(n-1)$, and solving this recursion 
from $S_p(0)=0$ and $S_p(1)=1$, we obtain $ S_p(n)=(1-(p-1)^n)/(2-p)$. Hence,
\[
\max_{m\geq n}S_p(m)\leq \bar{S}_p(n):=\frac{1+(1-p)^n}{2-p}
\]
for each $n\in\mathbb{N}$, and using the previous inequality for $\bar{C}_p(n)$ it follows 
that
\begin{equation}\label{eq:barreccorrelation}
\bar{C}_p(n)\leq \bar{C}_p(\lfloor n/2\rfloor)\left(|f(p)|\bar{S}_p(m)+
                   |g(p)|\bar{S}_p(m-1)+|h(p)|\bar{S}_p(m-2)\right) .
\end{equation}

\medskip\noindent
We have three cases:

\begin{itemize}
\item[\emph{a)}] 
For $0< p\leq 1/3$ we have $f(p)\leq 0 \leq g(p), h(p)$. In this range~\eqref{eq:barreccorrelation} becomes 
\[
\bar{C}_p(n)\leq \bar{C}_p(\lfloor n/2\rfloor)\left((1-2p)+ \frac{2(1-p-p^2)}{2-p}(1-p)^n\right). 
\]
\item[\emph{b)}] For $1/3\leq p \leq 1/2$ we have $f(p), g(p)\leq 0\leq h(p)$, hence
\[
\bar{C}_p(n)\leq \bar{C}_p(\lfloor n/2\rfloor)\left(\frac{p(3-4p)}{2-p}+2p(1-p)^n\right).
\]
\item[\emph{c)}] Finally, for $1/2\leq p < 1$ we have $g(p)\leq 0\leq f(p), h(p)$, and so
\[
\bar{C}_p(n)\leq \bar{C}_p(\lfloor n/2\rfloor)\left(\frac{p}{2-p}+\frac{2p(p+1)}{2-p}(1-p)^n\right).
\]
\end{itemize}

\noindent In all three cases, the inequality has the form
\[
\bar{C}_p(n)\leq \bar{C}_p(\lfloor n/2\rfloor)\left(\eta_p+A_p\,(1-p)^n\right),
\]
with $\eta_p\in (0,1)$ and $A_p>0$. For $n\in\mathbb{N}$, let $k=k(n)$ be such that
$2^k\leq n < 2^{k+1}$. Iterating the previous inequality, and taking into account that 
$\bar{C}_p(1)\leq 1/4$ we obtain
\[
\bar{C}_p(n)\leq \bar{C}_p(2^k) \leq \frac{\eta_p^k}{4}\,\prod_{j=1}^k
\left(1+ A_p\eta_p^{-1}\,(1-p)^{2^j}\right).
\]
Since $2^{k+1}>n$ then $k > \log(n)/\log(2)-1$, it follows that
\[
|C_p(n)|\leq \bar{C}_p(n)\leq \frac{\eta_p^{\log(n)/\log(2)-1}}{4}\,
             \prod_{j=1}^k \left(1+A_p\,\eta_p^{-1}\,(1-p)^{2^j}\right),
\]
and the result follows with $\alpha_p:=-\log(\eta_p)/\log(2)$ and
\[
K_p:=\frac{1}{4\eta_p}\prod_{j=1}^\infty \left(1+A_p\,\eta_p^{-1}\,(1-p)^{2^j}\right)
    \leq \frac{\exp\left(A_p\,(p\,\eta_p)^{-1}\right)}{4\eta_p}<\infty.
\]
\rightline{$\Box$}

\medskip
\subsection{}\label{subsec:auxiliars} This section is devoted to the proof of 
Propositions~\ref{prop:sumnuptail}, \ref{prop:correlationtail} and \ref{lem:QR}.

\medskip

\begin{proposition}~\label{lem:Ipdeq}
For each $p\in (0,1)$ the function
\[q\mapsto I_p(q):=q/(q+1)\log\left(q/(1-p)\right)+(1-q)/(q+1)\log\left((1-q)/p\right)\]
is non-negative, strictly convex, and satisfies
\[ \min\{I_p(q): q\in (0,1)\}=I_p(1-p)\equiv 0. \]
\end{proposition}

\medskip
\begin{proof}
Since $x\mapsto -\log(x)$ is a concave function,
then
\begin{eqnarray*}
I_p(q)
&  = &\frac{1}{q+1}\left(-q\,\log\left(\frac{1-p}{q}\right)-
                    (1-q)\log\left(\frac{p}{1-q}\right)\right)\\
&\geq&\frac{1}{q+1}\left(-\,\log\left(q\frac{1-p}{q}+(1-q)\frac{p}{1-q}\right)\right)=0.
\end{eqnarray*}
On the other hand,
\[
\frac{dI_p(q)}{dq}     =
\frac{1}{(q+1)^2}\left(\log\left(\frac{q}{1-p}\right)-2\log\left(\frac{1-q}{p}\right)\right)
                      =0\Leftrightarrow q=1-p.
\]
In this way we prove that $I_p$ is non-negative with minimum at $q=1-p$.
Now,
\[
\frac{d^2I_p(q)}{dq^2} =
\frac{1}{q(1-q^2)}+\frac{2}{(q+1)^3}\left(2\log\left(\frac{1-q}{p}\right)+
                        \left(\frac{1-p}{q}\right)\right) > 0
\]
for all $p,q\in (0,1)$. For this, note that if $1-q\geq p$ then $1-p\geq q$, and in this case
we have
\[
\frac{d^2I_p(q)}{dq^2} \geq \frac{1}{q(1-q^2)} > 0,
\]
otherwise, for $1-q < p$ then $1-p < q$, and taking into account that $-\log(x)\geq 1-x$, we
obtain
\begin{eqnarray*}
\frac{d^2I_p(q)}{dq^2}& =  & \frac{1}{q(1-q^2)}-\frac{2}{(q+1)^3}
                        \log\left(\left(\frac{(1-q)^2}{p}\right)\frac{(1-p)}{q}\right)\\
                     &\geq& \frac{1}{q(1-q^2)}+\frac{2}{(q+1)^3}
                           \left(1-\left(\frac{(1-q)^2}{p}\right)\frac{(1-p)}{q}\right)\\
                     &\geq& \frac{1}{q(q+1)}\left(\frac{1}{1-q}-2\frac{(1-q)q}{(q+1)^2}\right)
                           =\frac{1+3q^2}{q(q+1)^3(1-q)} > 0.
\end{eqnarray*}
Therefore $I_p$ is strictly convex.
\end{proof}

\medskip
\begin{proposition}~\label{lem:Stirling}
For $p\in (0,1)$ and  $k, n\in \mathbb{N}$ let
\[ \nu_p(k,n):= \binom{k-1}{n-k}(1-p)^{n-k}p^{2k-n-1}.  \]
Then, for each $n\geq 3$ there are constants $0\leq A^-\leq A^+\leq1$ such that
\[ e^{-n\,I_p(n/k-1)}\,A^{-} \leq\nu_p(k+1,n+1)\leq e^{-n\,I_p(n/k-1)}\,A^{+}, \]
for all $\lfloor n/2\rfloor \leq k\leq n$.
\end{proposition}

\begin{proof}
\noindent A very useful refinement of Stirling's approximation, first published 
in~\cite{robbins1955}, states that
\[
\sqrt{2\pi\,n}\,n^n\,\exp\left(-n + \frac{1}{12n+1}\right)
\leq n!\leq
\sqrt{2\pi\,n}\,n^n\,\exp\left(-n + \frac{1}{12n}\right)
\]
for all $n\in\mathbb{N}$. Hence, for each $p\in (0,1)$, $n\geq 3$, and $[n/2] < k <n$,
we have
\[\exp\left(-\epsilon_{n,k}\right) \leq 
  \frac{\sqrt{2\pi\,k\,(n/k-1)\,(2-n/k)}}{k^k\,(n-k)^{-(n-k)}\,(2k-n)^{-(2k-n)}}
  \times\binom{k}{n-k} \leq \exp\left(+\epsilon_{n,k}\right),\]
with $\epsilon_{n,k}=(4\,\min(n-k,2k-n))^{-1}$. A simple computation shows that
\[k^k\,(n-k)^{-(n-k)}\,(2k-n)^{-(2k-n)}(1-p)^{n-k}p^{2k-n}=\exp\left(-n\,I_p(1-k/n)\right),\]
with  $I_p(q):=q/(q+1)\log\left(q/(1-p)\right)+(1-q)/(q+1)\log\left((1-q)/p\right)$ for 
each $q\in(0,1)$. Hence
\begin{equation}\label{eq:stirling}
\frac{e^{-n\,I_p(n/k-1)-\epsilon_{n,k}}}{\sqrt{2\pi\,k\,(n/k-1)(2-n/k)}} 
\leq
\nu_p(k+1,n+1)
\leq
\frac{e^{-n\,I_p(n/k-1)+\epsilon_{n,k}}}{\sqrt{2\pi\,k\,(n/k-1)(2-n/k)}},
\end{equation}
for each $n\geq 3$ and $\lfloor n/2\rfloor\leq k\leq n$. On the other hand,
\begin{eqnarray*}
\nu_p(n/2+1,n+1)&=& (1-p)^{n/2}\equiv \lim_{q\to 1}\exp(-n\,I_p(q)),\\
\nu_p(n+1,n+1)  &=& p^n\equiv \lim_{q\to 0}\exp(-n\,I_p(q)).
\end{eqnarray*}
Thus, using
\[A^{\pm}:=\left\{\begin{array}{lr}  
 e^{\pm \left(4\,\min(n-k,2k-n)\right)^{-1}}\left(2\pi\,k\,(n/k-1)(2-n/k)\right)^{-1/2}
                                   & \text{ if }\ n/2< k < n,\\
                                 1 & \text{ otherwise,}
           \end{array}\right.\]
we can extend~\eqref{eq:stirling} to
\[e^{-n\,I_p(n/k-1)}\,A^{-} \leq\nu_p(k+1,n+1)\leq e^{-n\,I_p(n/k-1)}\,A^{+} \]
which holds for all $n\geq 3$ and $\lfloor n/2\rfloor \leq k\leq n$. Finally, a simple 
computation shows that
\[A^{+}\leq \max\left(1, 
\frac{e^{1/4}\sqrt{n-1}}{\sqrt{2\pi(n-2)}},\frac{e^{1/4}\sqrt{n+2}}{\sqrt{4\pi(n-2)}}\,
                \right)=1,\]
for $n\geq 3$.
\end{proof}

\medskip\noindent{{\bf{\em Proof of Proposition~\ref{prop:sumnuptail}}.}
We have already proved that the function $q\mapsto I_p(q)$ is
non-negative, vanishes only at $q=1-p$, and is strictly convex. Hence,  Taylor's Theorem
ensures that
\[
(q-(1-p))^2\times\min_{|q-(1-p)|\leq \epsilon}\left.\frac{d^2I_p}{dq^2}\right|_{q}
\leq I_p(q)\leq
(q-(1-p))^2\times\max_{|q-(1-p)|\leq \epsilon}\left.\frac{d^2I_p}{dq^2}\right|_{q}.
\]
Since $d^2I_p/dq^2|_{q=1-p}=(p(1-p)(2-p)))^{-1}$ and $q\mapsto d^2I_p/dq^2$ is continuous,
then, for all $\alpha \in (0,1)$ there exists $\epsilon_\alpha >0$ such that
\[
\frac{\alpha\,(q-(1-p))^2}{p(1-p)(2-p))} \leq I_p(q)
\leq \frac{\alpha^{-1}\,(q-(1-p))^2}{p(1-p)(2-p))}
\]
for all $q\in (1-p-\epsilon_\alpha,1-p+\epsilon_\alpha)$. With this, and using 
Proposition~\ref{lem:Stirling}, it follows that
\[
0\leq \sum_{|k/n-(2-p)|\geq \epsilon}\nu_p(k+1,n+1)\leq
n\times\exp(-nI_p(\epsilon))
\leq n\times \exp\left(\frac{-n\,\alpha\,\epsilon_\alpha^2}{p(1-p)(2-p)}\right)
\]
for all $\epsilon\leq \epsilon_\alpha$ and $n\geq 3$. By taking $n$ such that
$d_\beta(n) =\sqrt{p(1-p)(2-p)(\beta+1)\log(n)/n}\leq \epsilon_\alpha$ we obtain
\[
0\leq \sum_{|k/n-(2-p)|\geq d_\beta(n)}\nu_p(k+1,n+1)\leq n^{1-\alpha(\beta+1)},
\]
and the claim follows if we fix $\alpha=(b+\beta+2)/(2\beta+2)$
and $n_\beta$ such that $d_\beta(n)\leq \epsilon_\alpha$ for all $n\geq n_\beta $.

\rightline{$\Box$}

\medskip
\noindent{\bf{\em Proof of Proposition~\ref{prop:correlationtail}}.}
Fix $n\geq 5$ and $p\in (0,1)$. Since $C_p(k) \leq 1/4$ for all $k\in \mathbb{N}$ and 
$\nu_p(n,k) >0$, then 
\begin{eqnarray*}
\left|\sum_{k > 2 n/ 3}\mathcal{W}_p(k,n)C_p(k)\right| 
    & \leq &\frac{1}{4}\,\sum_{k > 2n/3}\left| \mathcal{W}_p(k,n)\right|\\
    &\leq  &\frac{n\left(\left| f(p) \right|  + \left| g(p) \right| +
             \left| h(p) \right|\right)}{12}
             \left(\max_{2 n / 3 < k,\, n-2\leq m\leq n}\nu_p(k,m)\right).
\end{eqnarray*}
Then, using Proposition~\ref{lem:Stirling} and the fact that $\left| f(p) \right|  + 
\left| g(p) \right| +\left| h(p) \right| \leq 2$ for all $p \in (0,1)$, we obtain,
\[ \left|\sum_{k > 2n/3}\mathcal{W}_p(k,n)C_p(k)\right|  
\leq 
\frac{n}{6} \exp\left( -(n-2)\min_{2 n / 3 < k,\, n-2\leq m\leq n}\, 
                 I_p\left( (m-1)/(k-1) - 1\right)\right) 
\]
for each $n\geq 5$.
Now, according to Proposition~\ref{lem:Ipdeq}, the function $q\mapsto I_p(q)$ is 
monotonously decreasing in $[0,1-p]$ and monotonously decreasing in $[1-p,1]$. Since for 
each $p\in (0,p^*)$ and $n\geq 5$ we have $0 < (m-1)/(k-1)-1 < 1/2-3/2n < 1/2 \leq 1 - p^*$, 
for each $n-2\leq m\leq n$ and $k\geq 2 n / 3$. Hence, 
\[ \min_{2 n / 3 < k,\, n-2\leq m\leq n}\, I_p\left( (m-1)/(k-1) - 1 \right) 
    > I_p(1/2)  = \frac{1}{3}\log\left(\frac{1}{4p(1-p)}\right).\]
which leads to the desired result, namely
\[\left|\sum_{k > 2(n-1)/3}\mathcal{W}_p(k,n)C_p(k)\right|  
            \leq \frac{n\left(4p(1-p)\right)^{(n-2)/3}}{6}.\]

\rightline{$\Box$}


\medskip \noindent {\bf{\em Proof of Proposition~\ref{lem:QR}}}.
A straightforward computation shows that
\[
\lambda_p\,\phi^{j}(\lambda_p^{k-1}\,x)=\ell(\lambda_p^k\,x) \text{ and } 
\lambda_p\,\psi^{j}(\lambda_p^{k-1}\,x)=u(\lambda_p^k\,x),
\]
for all $x\geq e$, $k\in \mathbb{N}$ and $1\leq j\leq k$. 
It is easily checked that, for each $p\in (0,1)$ and $\beta > b$, there exists $x_1\geq e$ 
such that both $\ell$ and $u$ are increasing functions in $[x_0,\infty)$. 

\medskip\noindent
For each $p\in(0,1)$ and $x\geq x_1$ let $Q(x)$ be the largest solution to
\[
Q(x)=\exp\left(-\frac{d(x)}{\sqrt{Q(x)}\,\lambda_p}
\sum_{m=0}^{\infty}\sqrt{\frac{m+1}{\lambda_p^m}}\right).
\]
It is not difficult to check that $Q(x)\in(0,1)$ and since $d(x)\rightarrow 0$ as 
$x\to\infty$, then $Q(x)\rightarrow 1$ as $x\to\infty$.

\medskip\noindent Now, fix $k\in\mathbb{N}$, and $x_0\geq x_1$ so large that 
$\lambda_p\,Q(x) \geq 1$ for all $x\geq x_0$. Let $Q_{k,0}(x):=1$, and define recursively
\[
Q_{k,j+1}(x):=\frac{Q_{k,j}(x)}{1+\lambda_p^{-(k-j)/2-1}\,Q(x)^{-1/2}\,\sqrt{k-j+1}\,d(x)}
\]
for $0\leq j\leq k-1$. Clearly 
\begin{eqnarray*}
1\geq 
Q_{k,j}(x)&=&\prod_{i=0}^{j-1}
\left(1+\lambda_p^{-(k-i)/2-1}Q(x)^{-1/2}\,\sqrt{k-i+1}\,d(x)\right)^{-1}\\
&\geq& \exp\left(-\frac{d(x)}{\sqrt{Q(x)}}
               \sum_{i=0}^{j-1}\sqrt{k-i+1}\,\lambda_p^{-(k-i)/2-1}\right)\\
&\geq& \exp\left(-\frac{d(x)}{\sqrt{Q(x)}\,\lambda_p}
\sum_{m=0}^{\infty}\sqrt{\frac{m+1}{\lambda_p^m}}\right)=Q(x).
\end{eqnarray*}

\medskip \noindent Since $x\geq e > \lambda_p$, then $ x^{k+1}\geq \lambda_p^k\,x$, 
and therefore
\begin{eqnarray*}
d(\lambda_p^k\,x)&=&\sqrt{\frac{p(1-p)(2-p)(\beta+1)\log(\lambda_p^k\,x)}{\lambda_p^k\,x}}\\
                &\leq&\lambda_p^{-k/2}\sqrt{\frac{p(1-p)(2-p)(\beta+1)(k+1)\log(x)}{x}}
                 =\sqrt{\frac{k+1}{\lambda_p^k}}\,d(x).
\end{eqnarray*}
Hence, for $j=1$ we have
\[
\ell(\lambda_p^k\,x)=\frac{\lambda_p^{k-1}\,x}{1+d(\lambda_p^k\,x)/\lambda_p}
                    \geq 
                    \frac{\lambda_p^{k-1}\,x}{1+\lambda_p^{-k/2-1}\sqrt{k+1}\,d(x)}
                    =Q_{k,1}(x)\,\lambda_p^{k-1}\,x.
\]   
Suppose that $\ell^j(\lambda_p^k\,x)\geq Q_{k,j}(x)\,\lambda^{k-j}\,x$ for $j<k$. 
Since $Q_{k,j}(x)\,\lambda_p^{k-j}\,x \leq \lambda_p^{k-j}\,x \leq x^{k-j+1}$, then
\begin{eqnarray*}
d(Q_{k,j}\,\lambda_p^{k-j}\,x)&=&\sqrt{\frac{p(1-p)(2-p)(\beta+1)
                                  \log(Q_{k,j}(x)\,\lambda_p^{k-j}\,x)
                                          }{\lambda_p^{k-j}\,Q_{k,j}(x)\,x}}\\
                              &\leq&\lambda_p^{-(k-j)/2}Q_{k,j}(x)^{-1/2}\sqrt{k-j+1}\,d(x)\\
                              &\leq& \lambda_p^{-(k-j)/2}Q(x)^{-1/2}\sqrt{k-j+1}\,d(x). 
\end{eqnarray*}
Taking into account that $y\mapsto \ell(y)$ is an increasing function for $y\geq x_0$, and 
since $Q_{k,j}(x)\,\lambda_p^{k-j}\,x\geq Q(x)\,x\geq x_0$, then
\begin{eqnarray*}
\ell^{j+1}(\lambda^k_p\,x)&\geq& 
\frac{Q_{k,j}(x)\,\lambda_p^{k-j}\,x}{\lambda_p+d(Q_{k,j}(x)\,\lambda_p^{k-j}\,x)}\\
                           &\geq& 
\frac{Q_j(x)\,\lambda_p^{k-j-1}\,x}{1+\lambda_p^{(k-j)/2-1}\,Q(x)^{-1/2}\,
                                                                 \sqrt{k-j+1}\,d(x)}\\
                           &= & Q_{k,j+1}(x)\,\lambda_p^{k-j-1}\,x.
\end{eqnarray*}
In this way we have proved that
\[
\ell^{j}(\lambda^k_p\,x)\geq Q_{k,j}(x)\,\lambda_p^{k-j}\,x\geq Q(x)\,\lambda_p^{k-j}\,x,
\]
for all $k\in\mathbb{N}$ and $1\leq j\leq k$, and $x\geq x_0$. 

\medskip\noindent By taking $R=R(x)$ the smallest solution to the equation
\[
R=\exp\left(\frac{d(x)}{\sqrt{R}\,\lambda_p}
                        \sum_{m=0}^{\infty}\sqrt{\frac{m+1}{\lambda_p^m}}\right),
\]
the previous argument can be easily adapted to deduce 
\[
u^{j}(\lambda^k_p\,x)\leq R(x)\,\lambda_p^{k-j}\,x,
\]
for all $k\in\mathbb{N}$ and $1\leq j\leq k$, and $x\geq x_0$.

\rightline{$\Box$}

\medskip\subsection{}\label{subsec:lower}
{\bf{\em Proof of Lemma~\ref{theo:lowerbound}}.}
Suppose that $n^*\geq 5$ and $b > 0$ are such that $C_p(k)\geq n^{-b}$ for all 
$\lfloor n^*/2\rfloor\leq  k < n^*$. Using~\eqref{eq:reccorrelation}, and using 
Proposition~\ref{lem:Wp>0} and Proposition~\ref{prop:correlationtail}, we obtain
\begin{eqnarray*}
C_p(n)&\geq& \left(\sum_{ k < 2(n-1)/3}\mathcal{W}_p(k,n)\right)
     \left(\min_{ k < 2(n-1)/3}C_p(k)\right)-\frac{n}{6}\left(4p(1-p)\right)^{(n-2)/3} \\
      &\geq&\left(\sum_{k \leq n}\mathcal{W}_p(k,n)\right)
     \left(\frac{2(n-1)}{3}\right)^{-b}-\frac{n}{6}\left(4p(1-p)\right)^{(n-2)/3},         
\end{eqnarray*}
for $n^*\leq n < (3n^*+1)/2$. Since 
$\sum_{k=\lfloor n/2\rfloor}^n \nu_p(k,n)=(1-(p-1)^n)/(2-p)$, then 
$\sum_{k \leq n}\mathcal{W}_p(k,n)=(1-2p)(2-3p)/(2-p)-2p(p-1)^n$, 
for all $n\geq 2$. From this it follows that
\[
C_p(n)\geq n^{-b} \left(\left( \frac{(1-2p)(2-3p)}{2-p}-2p(p-1)^n \right) 
                  \left(\frac{3n}{2}\right)^{b}
               -\frac{n^{1+b}\left(4p(1-p)\right)^{(n-2)/3} }{6}  \right),
\]
for $n^*\leq n < (3n^*+1)/2$. Now, if $n^*$ is large enough so that 
\[
\left( \frac{(1-2p)(2-3p)}{2-p}-2p(p-1)^n \right) 
\left(\frac{3}{2}\right)^{b} -\frac{n^{1+b}\left(4p(1-p)\right)^{(n-2)/3} }{6}  > 1
\]
for all $n\geq n^*$, then we have $C_p(n)\geq n^{-b}$ for 
$\lfloor n^*/2\rfloor\leq n < (3n^*+1)/2$. In this way, the interval where the lower bound 
$C_p(n)\geq n^{-b}$ holds, enlarges from $\lfloor n^*/2\rfloor\leq n < n^*$ to
$\lfloor n^*/2\rfloor\leq n < (3n^*+1)/2$. This constitutes the induction step from which 
it follows that $C_p(n)\geq n^{-b}$ for all $n\geq \lfloor n^*/2\rfloor$ and $p\in (0,p^*)$. 
The validity of this induction depends on the existence of $n^*$ and $b$ such that
$C_p(k)\geq n^{-b} $ for all  $\lfloor n^*/2\rfloor\leq  k < n^*$  and all $p\in(0,p^*)$, 
and 
\[
\left( \frac{(1-2p)(2-3p)}{2-p}-2p(p-1)^n \right) 
\left(\frac{3}{2}\right)^{b_p}-\frac{n^{1+b_p}\left(4p(1-p)\right)^{(n-2)/3} }{6}  > 1
\] 
for all $n\geq n^*$ and all $p\in(0,p^*)$. We found such $b$ and $n$ by 
using~\eqref{eq:explicitrec}, which allows us to explicitly compute $C_p(n)$ as function 
of $p$. By doing so, we verify that $C_p(n)\geq n^{-3}$ for $n\leq 402$ and
$p\in (0,p^*)$. We also verified that the function 
\[
p\mapsto \left(\frac{(1-2p)(2-3p)}{2-p}-2p(p-1)^n \right) 
         \left(\frac{3}{2}\right)^{3} -\frac{n^{4}\left(4p(1-p)\right)^{(n-2)/3} }{6}
\]
increases monotonously with $n$ for all $n\geq 60$. Since 
\[
\min_{p\leq p^*}B_p(402)=B_{p^*}(402)\approx 1.0015,
\]
then we can choose $n^*=402$ and $b=3$. The theorem follows with $b=3$ and 
$n_{p^*}=\lfloor n^*/2\rfloor=201$.
										 
\rightline{$\Box$}

\section{Conclusions.}

\medskip\subsection{Summary.} 
We approach the expansion-modification system considering the action of the 
expansion-modification dynamics over the space of Borel measures on 
$\{0,1\}^{{\mathbb N}_0}$. Inside this framework, we have proved the existence and 
uniqueness of the stationary measure, which attracts all initial distributions as times 
goes to infinity~(Theorem~\ref{theo:asymptotic}). We have also proved that the 
stationary measure has correlations decaying faster than a certain power-law 
(Theorem~\ref{theo:upperbound}). Although the calculations leading to these results
strongly relay on a recurrence relation very specific to this system, some degree of
generality could be expected. For instance, existence and uniqueness of the stationary 
measure depends only on the primitivity of the stochastic matrices describing the dynamics 
of the finite-size marginals. Following a simple argument (see Lemma 2.4.4. 
in~\cite{2012Koslicki}), one can easily prove that the primitivity of the matrix associate 
to the $1$-marginal implies the primitivity of the matrices associated to all marginals, 
from which it follows the *-weak convergence of the distributions towards a unique 
stationary measure. 
On the other hand, decay of correlations and asymptotic scaling of the correlation function
could be proved  for cases where recurrence relations similar to~\eqref{eq:reccorrelation} 
are satisfied. For this kind of increasingly complex recurrence relations, which cannot 
be solved in closed form, one cannot expect to obtain an exact scaling behavior. 

\medskip \noindent The main contribution of this work is the rigorous proof of the 
asymptotic scaling of the correlation function, for a rather large interval of mutation 
probabilities (Theorem~\ref{theo:asymptoticscaling}). 
In order to prove this it was necessary to find a power law bounding from below the 
correlation function (Theorem~\ref{theo:lowerbound}). 
The validity of is result, and consequently of our main theorem, could in principle be 
extended to the range $0<p<1/2$, where it seems that $C_p(n)>0$ for all $n$ sufficiently 
large. However, our technique depends on the explicitly computation of $p\mapsto C_p(n)$ 
for values of $n$ in the range were $C_p(n)>0$, and we observe that this range diverges 
as $p$ approaches $1/2$. Hence, the choice of the range $(0,p^*)$ is rather arbitrary and 
of practical nature.

\medskip \noindent \subsection{Scaling.} 
The argument developed in Section~\ref{sec:heuristic} suggests that the scaling property 
extends to the whole range of mutation probabilities and that  the corresponding scaling 
exponent varies piece-wise smoothly as function of that parameter. The same argument allows 
us to furnish an explicit expression for the scaling exponent that fits very well the 
exponent calculated by numerically solving the recurrence relation~\eqref{eq:reccorrelation}. 

\medskip \noindent 
The scaling behavior of a system is usually detected by observing a power law behavior in 
its power spectrum. It can be easily verified that a power law behavior in the correlation
function implies a power law behavior for the power spectrum 
$f(\omega):=\left|{\mathcal F}(C_p)(\omega)\right|$ where 
${\mathcal F}(C_p)$ denotes the discrete Fourier transform of the correlation function. 
In our case, a straightforward computation shows that 
$f(\omega)={\mathcal O}(\omega^{-\alpha_p})$ with
\[
\alpha_p:=1-\beta_p=\frac{\log(2-p)-\log(1-2p)-\log(2-3p)}{\log(2-p)}.
\]

\medskip\subsection{Related work.} 

\medskip\noindent In~\cite{messer&al2005}, Messer, Arndt and L\"assing
study a model generalizing the expansion-modification systems for which they deduce an
asymptotic scaling behavior and a closed expression for the scaling exponent. Although we 
have followed similar ideas, their model evolves in continuous time. Besides, ours are the 
first rigorous results concerning those kind of models.
In~\cite{mansilla2000multiscaling}, Mansilla and Cocho analyze the correlation function of 
the model we consider, and they claim the existence of several scaling exponents in the 
non-asymptotic regime. 

\medskip\noindent In a more recent paper, the expansion-modification system was studied 
in relation to the universality of the rank-ordering distributions~\cite{alvarez2010order}. 
The authors numerically  found an order-disorder transition which would manifest itself on 
the scaling behavior. According to them, there would be a critical mutation probability 
$p_c\approx 0.4$, such that for $p>p_c$, long-range correlations and consequently 
the scaling behavior of $C_p$ would disappear. As we have shown, this kind of 
order-disorder transition does not occur. The apparent drop of long-range correlations for 
large $p$ can be explained by the lack of statistics. Indeed, a huge amount of data is 
needed to empirically compute correlation functions with a fast power-law decay. In order 
to observe a power-law decay with exponent $-5$ up to two decades, one would need of the 
order of $10^{10}$ sample sequences in $\{0,1\}^{100}$, obtained by the action of the 
expansion-modification dynamics over an arbitrary seed, after a sufficiently 
long transient. This explains why the scaling behavior of the  $C_p$ is very difficult 
to observe from empirical computations for $p\geq 0.4$, where $\beta_p > 5$ 
(see Figure~\ref{fig:comparisonbetas}). 

\medskip \noindent
Some other models of random substitutions have been previously studied. 
In~\cite{1989GodrecheLuck}, Godr\`eche and Luck 
considered a ``perturbed'' Fibonacci substitution for which they show that the Fourier 
spectrum is of mixed type, \emph{i.e}., it contains both singular and continuous parts. 
Zaks obtain the same for a substitution system which can be seen a perturbation of the 
Thue-Morse sequence~\cite{zaks2001}. In both cases the system can be seen as a random 
perturbation of a quasicrystal. These is not the case for the expansion-modification system 
for which we observe a continuous spectrum.

\medskip\noindent
Random substitutions have been treated in some generality by Koslicki in~\cite{2012Koslicki}, 
in the framework of countable Markov chains. We approach the expansion\-mo\-di\-fi\-ca\-tion
system from another point of view, considering the action of the expansion-modification 
dynamics over the space of Borel measures on $\{0,1\}^{{\mathbb N}_0}$. 
Nevertheless, we could in principle use Koslicki's results to, for instance, derive our 
Theorem~\ref{theo:asymptotic}. More interestingly, its result ensuring the almost sure
convergence of frequencies (Theorem 2.4.10 in~\cite{2012Koslicki}) can be used to prove that 
the stationary measure $\mu_p$ is absolutely continuous with respect to a shift-invariant 
ergodic measures. More closely related to ours is the approach by Toom and coworkers 
(see~\cite{2011Toom} and references therein). 
They define a certain class of substitution operators acting on shift-invariant measures, 
nevertheless, the class of substitutions they consider do not include the 
expansion-modification system. None of the above mentioned approaches  
directly apply in our setting, therefore some technical work would be required to adapt 
their results to our case.

\bigskip
\bibliographystyle{plain}

\begin{thebibliography}{99}

\bibitem{alvarez2010order} R. Alvarez-Mart\'\i nez, G. Mart\'\i nez-Mekler and G. Cocho, 
``Order-disorder transition in conflicting dynamics leading to rank-frequency generalized 
beta distributions'' \emph{Physica A} {\bf 390} (1) 120--130 (2011).

\bibitem{buldyrev1995long} A. L. Buldyrev, A. L. Goldberger, S. Havlin, R. N. Mantegna, 
M. E. Matsa, C.-K. Peng, M. Simons, H. E. Stanley,   ``Long-range correlation properties 
of coding and noncoding DNA sequences: GenBank Analysis'', \emph{ Phys. Rev. E } {\bf 51}, 
5084--5091 (1995).

\bibitem{1989GodrecheLuck} C. Godr\`eche and J. M. Luck,  ``Quasiperiodicity and Randomness 
in Tilings of the Plane'', \emph{ Journal of Statistical Physics} {\bf 55} (1/2) 1--28 (1989).

\bibitem{2012Koslicki} D. Koslicki, ``Substitution Markov Chains with Applications to 
Molecular Evolution'', Ph. D. Dissertation, Pennsylvania State University, 2012.

\bibitem{li1989spatial} W. Li, ``Spatial l/f spectra in open dynamical systems'', 
\emph{ Europhysics Letters} {\bf 10} (5) 395--400 (1989).

\bibitem{li1991expansion} W. Li, ``Expansion-modification systems: A model for spatial 
$1/f$ spectra'', \emph{ Physical Review A} {\bf 43} (10) 5240--5260 (1991).  

\bibitem{li1992long} W. Li and K. Kaneko, ``Long-range correlation and partial $1/f$ a 
spectrum in a noncoding DNA sequence'', \emph{ Europhysics Letters } { \bf 17}, 655--660 (1992).

\bibitem{li1997study} W. Li, ``The study of correlation structures of DNA sequences: 
a critical review'', \emph{ Computers Chem. }{\bf 21}, 257--271 (1997).


\bibitem{ma2008infinite} J. Ma, A. Ratan, B. J. Raney, B. B. Suh, M. Miller, D. Haussler, 
``The infinite site model of genome evolution'', \emph{ Proc. Nat. Acad. Sci. } {\bf 105}, 
14254--14261 (2008). 

\bibitem{mansilla2000multiscaling} R. Mansilla and  G. Cocho, ``Multiscaling in  
expansion-modification systems: An explanation for long range correlation in DNA'', 
\emph{ Complex Systems} {\bf 12 }, 207--240 (2000).

\bibitem{mekler2009universality} G. Mart\'{\i}nez-Mekler, R. Alvarez-Mart\'\i nez, 
M. Beltr\'an del R\'{\i}o, R. Mansilla, P. Miramontes, G. Cocho, ``Universality of 
rank-ordering distributions in the arts and sciences'', \emph{ Plos One} {\bf 4}, 
e4791, 1--7 (2008). 

\bibitem{messer&al2005} 
Ph. W. Messer, P. F. Arndt, and M. L\"assig, ``Solvable sequence evolution models 
and genomic correlations'', \emph{ Phys. Rev. Lett.} {\bf 94}, 138103 (2005).

\bibitem{peng1992} C.-K. Peng, S. V. Buldyrev, A. L. Goldberger, S. Havlin, F. Sciortino, 
M. Simons, and H. E. Stanley, ``Long-range correlations in nucleotide sequences'', 
\emph{ Nature} {\bf 356} 168--179 (1992).

\bibitem{robbins1955} H. Robbins, ``A Remark on Stirling's Formula'', 
\emph{ The American Mathematical Monthly} {\bf 62} (1) 26--29 (1955).

\bibitem{2011Toom} A. V. Rocha, A. B. Simas and A. Toom, ``Substitution Operators'',
\emph{Journal of Statistical Physics} {\bf 143}, 585--618.

\bibitem{saakian2008} D. B. Saakian, ``Evolution models with base substitutions, 
insertions, deletions, and selection'', \emph{ Phys. Rev. E} {\bf  78}, 0611920 (2008).

\bibitem{sobottka2011} M. Sobottka and A. G. Hart, ``A model capturing novel strand 
symmetries in bacterial DNA'', \emph{ Biochem. \& Biophys. Research Commun.} {\bf 410}, 
823--828 (2011).

\bibitem{zaks2001} M. Zaks, ``Multifractal Fourier spectra and power-law decay of 
correlations in random substitution sequences'', \emph{Physical Review E} {\bf 65}, 
011111 (2001).

\end{thebibliography}

\end{document}